\documentclass[12pt]{amsart}\usepackage{amsmath,amsthm,amsfonts,amssymb}
\usepackage{paralist,nicefrac}

\newcount\casecount
\newcommand{\case}[1]{\advance\casecount by 1\vskip .1in
    \goodbreak\noindent\textbf{Case \number\casecount:}\ #1\par\noindent}

\newcommand{\C}{\mathbb{C}}
\newcommand{\Z}{\mathbb{Z}}

\newcommand{\F}{\mathbb{F}}

\newcommand{\res}{\operatorname{res}}

\newcommand{\Stab}{\operatorname{Stab}}

\renewcommand{\sc}{\operatorname{sc}}
\newcommand{\rcc}{\operatorname{rcc}}

\newcommand{\Gal}{\operatorname{Gal}}

\newcommand{\Hom}{\operatorname{Hom}}
\newcommand{\Aut}{\operatorname{Aut}}

\newcommand{\ul}{\underline}

\newcommand{\Spin}{\mathrm{Spin}}

\newcommand{\SU}{\mathrm{SU}}
\newcommand{\PSU}{\mathrm{PSU}}

\newcommand{\SL}{\mathrm{SL}}
\newcommand{\PSL}{\mathrm{PSL}}
\newcommand{\Sp}{\mathrm{Sp}}

\newtheorem{thm}{Theorem}

\newtheorem{cor}[thm]{Corollary}

\newtheorem{prop}[thm]{Proposition}

\newtheorem{lem}[thm]{Lemma}

\newtheorem*{claim*}{Claim}
\begin{document}

\title{Beauville Surfaces and Finite Simple Groups}

\author{Shelly Garion}
\address{Max-Planck-Institut f\"ur Mathematik,
53111 Bonn, Germany}

\author{Michael Larsen}
\address{Department of Mathematics,
Indiana University,
Bloomington, IN 47405
U.S.A. }

\author{Alexander Lubotzky}
\address{Einstein Institute of Mathematics,
Edmond J. Safra Campus, Givat Ram,
The Hebrew University of Jerusalem,
Jerusalem, 91904
Israel }

\thanks{SG was supported by a European Postdoctoral Fellowship (EPDI)}
\thanks{ML was supported by grants from the National Science Foundation and the United States - Israel Binational Science Foundation.}
\thanks{AL was supported by grants from the European Research Council and the National Science Foundation.}

\maketitle

\begin{center}
\small
\textit{In memory of Fritz Grunewald, an inspiring mathematician and a dear friend}
\end{center}

\begin{abstract}
A Beauville surface is a rigid complex surface of the form $(C_1\times C_2)/G$,
where $C_1$ and $C_2$
are non-singular, projective, higher genus curves,
and $G$ is a finite group acting freely on the product.
Bauer, Catanese, and Grunewald conjectured that every finite simple group $G$,
with the exception of $A_5$, gives rise to such a surface. We prove that this is
so for almost all finite simple groups (i.e., with at most finitely many exceptions).
The proof makes use of the structure theory of finite simple groups,
probability theory, and character estimates.
\end{abstract}

\bigskip

Catanese \cite{Cat} defined a \emph{Beauville surface} to be
an infinitesimally rigid complex surface of the form $X := (C_1\times C_2)/G$,
where $C_1$ and $C_2$
are non-singular projective curves of genus $\ge 2$,
and $G$ is a finite group acting freely on the product.
Every $g\in G$ respects the product decomposition $C_1\times C_2$.
Let $G^0$ denote the subgroup of $G$ (of index $\le 2$) which preserves the
ordered pair $(C_1,C_2)$.  Any Beauville surface can be presented in such a
way that $G^0$ acts effectively on each factor.  Catanese called such a presentation
\emph{minimal} and proved that it is unique \cite[Proposition 3.13]{Cat}.
A Beauville surface is \emph{mixed} (resp. \emph{unmixed}) if
$[G:G^0]=2$ (resp. $G=G^0$).

There has been an effort to classify finite groups $G$ appearing in minimal presentations
of unmixed Beauville surfaces.
A finite group $G$ appears in this way if and only if
it admits an \emph{unmixed Beauville structure}.  This consists of
an ordered quadruple $(x_1,y_1,x_2,y_2)\in G$
such that each set $\{x_i,y_i\}$ generates $G$ and
\begin{equation*}
\Sigma(x_1,y_1)\cap \Sigma(x_2,y_2) = \{e\},
\end{equation*}
where $\Sigma(x,y)$ is the union of conjugacy classes of all powers of $x$, all powers of
$y$, and all powers of $xy$.
Bauer, Catanese, and Grunewald proved \cite[Theorem~7.16]{BCG}
that all sufficiently large alternating groups admit an unmixed Beauville structure and conjectured
\cite[Conjecture~7.17]{BCG} that all (non-abelian) finite simple groups except $A_5$ do so.
(The group $A_5$ is a genuine exception.)

Some progress has been made on this conjecture.  In \cite{FG} it was shown that it holds for alternating groups, i.e., that $A_n$ admits an unmixed Beauville structure for all $n\ge 6$.
Partial results for groups of the form $\PSL_2(q)$ and ${}^2B_2(2^{2f+1})$ were given in \cite{BCG}; complete results for the same two families together with  the Ree groups
${}^2G_2(3^{2f+1})$ were proved in \cite{FJ} and \cite{GP}.  Additionally, the groups of the form
$G_2(q)$ and ${}^3D_4(q)$ (in characteristic $p>3$) and $\PSL_3(q)$ and $\PSU_3(q)$
(all $q$) were treated in \cite{GP}.

The goal of the current paper is to prove the Bauer-Catanese-Grunewald conjecture for almost all finite simple groups.

\begin{thm}
\label{main}For every sufficiently large non-abelian finite simple group $G$, there exists a Beauville surface with minimal presentation $C_1\times C_2/G$.
\end{thm}

Before saying something about the proof, let us put our result in a more general group-theoretic context.  The existence of an unmixed Beauville structure on $G$ amounts to the realization of
$G$ as a quotient of (hyperbolic) triangle groups in two
``independent'' ways.  The question of which finite simple groups are quotients of triangle groups has attracted a lot of attention in the group theory literature.  There has been special interest in the $(2,3,7)$-triangle group, also known as the Hurwitz group (see \cite{Con} and the references therein), but considerable attention has been paid to other triangle groups as well (see \cite{Mar} and the references therein).  Usually the question has been asked from the angle of characterizing finite simple quotients of a fixed triangle group.  Here the point of view is different: given a finite simple group, we will provide several (in fact, many) ``triangle'' generations of it.  In this context, one should mention the old conjecture of Higman, proved by Everitt \cite{Eve}, asserting that for every
hyperbolic triangle group $\Gamma$, all but finitely many alternating groups are quotients of $\Gamma$.  Of course, this gives Theorem~\ref{main} for almost all the alternating groups.  One should not be tempted to conjecture that given $\Gamma$, almost all finite simple groups are quotients of $\Gamma$.  In this connection, \cite{Mar} has some intriguing and suggestive results.

Taking an even broader perspective, one can study finite quotients and representations of Fuchsian groups (see, e.g., \cite{LiSh} and \cite {LaLu}), but usually triangle groups are the most difficult to deal with.

Let us now give a short description of the proof.  As mentioned above, the result is known for alternating groups, and we can ignore the finitely many sporadic groups.  By the classification of finite simple groups, we need only consider groups of Lie type.  For each such group which is  sufficiently large, we will present two maximal tori $T_1$ and $T_2$.  If $C_i$ denotes the set of all conjugates of elements of $T_i$, we ensure that $C_1\cap C_2 = \{e\}$, and that in each $C_i$ we can find $x_i$ and $y_i$ such that
\begin{itemize}
\item[(a)] $G=\langle x_i,y_i\rangle$,
\item[(b)] $x_i y_i \in C_i$.
\end{itemize}
To assure (a), we use some of the well developed theory bounding the indices and the number of maximal subgroups in $G$.  On the other hand, (b) can be proved by means of character estimates.  While it is quite likely that the estimates we need can be deduced from Lusztig's theory, a much softer and more elementary method is presented in \S1.  This method works for general finite groups and may have some independent interest and applications.

The paper is organized as follows.  In \S1, the new character estimate is presented.  In \S2, we state Proposition~\ref{pairs}, which claims that two tori exist with some desired properties.  We illustrate the proposition by proving it for groups of the form $\PSL_n(q)$ for all but finitely many
pairs $(n,q)$.  We then present Theorem~\ref{main} modulo Proposition~\ref{pairs}, leaving the detailed case analysis required to prove Proposition~\ref{pairs} in general to \S3.
The only groups not covered by prior work are now the Ree groups ${}^2F_4(2^{2f+1})$, which are discussed in \S4, using a different method which is applicable to groups whose maximal subgroup structure is well understood.

This paper is dedicated to the memory of Fritz Grunewald, who
has been influential in the area of Beauville surfaces by connecting it to
group theory \cite{BCG}. It has been typical of Fritz to serve as a
bridge between different areas of mathematics. His legacy as a mathematician
and as an outstanding personality will be remembered for many years.

\subsection*{Acknowledgment} The first-named author would like to thank Ingrid Bauer, Fabrizio Catanese and Fritz Grunewald for introducing her to the fascinating world of
Beauville structures and for many useful discussions. She is also
grateful to Bob Guralnick and Eugene Plotkin for many helpful discussions.

\section{Elementary Character Bounds for General Finite Groups}

Let $G$ be any finite group.  We say that $a\in G$ is \emph{abstractly regular} if the centralizer
$Z(a)$ of $a$ in $G$ is abelian.  We give an upper bound for the absolute
value of irreducible characters evaluated at abstractly regular elements.

\begin{lem}
\label{transp}
Let $A$ denote a maximal abelian subgroup of $G$ and $N = N(A)$ its normalizer in $G$.
If $a\in A$ is abstractly regular, then $g^{-1} a g\in A$ if and only if $g\in N$.
\end{lem}

\begin{proof}
As  $A$ is abelian and $a\in A$, $A\subset Z(a)$.  As $a$ is abstractly regular, $Z(a)$ is abelian.
Since $A$ is maximal abelian, $A=Z(a)$.  Thus,
$$g^{-1} A g = g^{-1}Z(a)g = Z(g^{-1} a g).$$
If
$g^{-1} a g\in A$, then $Z(g^{-1} a g)$ contains $A$.  Since
$$|Z(g^{-1} a g)| = |Z(a) |= |A|,$$
we have
$$g^{-1} A g = g^{-1}Z(a)g = Z(g^{-1} a g) = A,$$
so $g\in N$.  The converse is trivial.
\end{proof}

\begin{thm}
\label{soft}
Let $A$ denote a maximal abelian subgroup of $G$.  Suppose that there
exist proper subgroups $A_1,\ldots,A_n$ of $A$ such that every element in
$$A^\circ:=A\setminus (A_1\cup A_2\cup \cdots \cup A_n)$$
is abstractly regular.  Let $N = N(A)$ denote the normalizer of $A$.  If $\chi$ is any irreducible
character of $G$, then
$$|\chi(a)| \le (4/\sqrt3)^n [N:A]$$
for all $a\in A^\circ$.
\end{thm}

\begin{proof}
Let $A^* = \Hom(A,\C^\times)$.  We identify $\Z A^*$ with the ring of characters of
virtual complex representations on $A$.  Let $\phi\in\Z A^*$ be the character associated
to the restriction of $\chi$ to $A$.  For each $i$ from $1$ to $n$, let $\phi_i\in A^*$ be
a character which is trivial on $A_i$ and non-trivial on $a$.
For every non-trivial root of unity $\omega$, there exists an integer $k$ such that
$|\omega^k-1|\ge \sqrt3$, so replacing $\phi_i$ by a character of the form $\phi_i^k$, we may assume
that $|\phi_i(a)-1|\ge \sqrt3$.
Let
$$\psi = \phi\prod_{i=1}^n (\phi_i - 1).$$

Let $g_1,\ldots,g_k$ denote a set of coset representatives for $G/N$.  By Lemma~\ref{transp},
distinct pairs $(g_i,a)$, with $a\in A^\circ$, give rise to distinct elements $g_i^{-1} a g_i$.
Therefore,
$$\sum_{a\in A^\circ} |\chi(a)|^2 =
\frac 1k \sum_{a\in A^\circ}\sum_{i=1}^k |\chi(g_i^{-1} a g_i)|^2 \le
\frac 1k\sum_{g\in G}|\chi(g)|^2 = \frac{|G|}k = |N|.$$
As $|\phi_i(a)-1|\le 2$ for all $i$,
$$\sum_{a\in A} |\psi(a)|^2 = \sum_{a\in A^\circ} |\psi(a)|^2
\le 4^n \sum_{a\in A^\circ} |\phi(a)|^2 = 4^n \sum_{a\in A^\circ} |\chi(a)|^2 \le 4^n |N|.$$

Writing $\psi$ as a linear combination $\sum c_i \psi_i$ of $\psi_i\in A^*$, we have
$$\sum_{a\in A} |\psi(a)|^2 = |A|\sum_i c_i^2.$$
Thus,
$$\sum_i c_i^2 \le 4^n [N:A].$$
As the $c_i$ are integers,
\begin{align*}|
\chi(a)| &= |\phi(a)|\le 3^{-n/2}|\psi(a)| = 3^{-n/2}\Bigm|\sum_i c_i \psi_i(a)\Bigm|
\le 3^{-n/2}\sum_i |c_i| \\
&\le 3^{-n/2}\sum c_i^2 \le (4/\sqrt3)^n [N:A].
\end{align*}

\end{proof}

\begin{cor}
\label{sl-tori}
Let $\chi$ be an irreducible character of $G = \SL_{r+1}(\F_q)$.  If $t_1\in G$ has an irreducible characteristic polynomial, then
\begin{equation}
\label{Singer}|
\chi(t_1)|  \le \frac{2(r+1)^2}{\sqrt3}.
\end{equation}
If $r\ge 2$ and the characteristic polynomial of $t_2\in G$ has an irreducible factor of degree $r$,  then
\begin{equation}
\label{other-SL}
|\chi(t_2)|  \le \frac{2r^2}{\sqrt3}.
\end{equation}
\end{cor}

\begin{proof}
Let $P(x)$ denote the characteristic polynomial of $t_1$.  Every non-zero vector $v\in \F_q^{r+1}$
determines an isomorphism
$$[Q(x)] \mapsto Q(t_1)(v)$$
of $\F_q$-vector spaces from $\F_q[x]/(P(x)) \cong \F_{q^{r+1}}$ to $\F_q^{r+1}$.
Thus we can identity $t_1$ with multiplication by a generating element $\alpha$ of the extension
$\F_{q^{r+1}}$ of $\F_q$.  The centralizer $A$ of $t_1$ consists of determinant~$1$ $\F_q$-linear automorphisms of $\F_{q^{r+1}}$ which commute with multiplication by $\alpha$ and therefore multiplication by all scalars in $\F_{q^{r+1}}$.  Thus $A =  \ker \F_{q^{r+1}}^\times \to \F_q^\times$.
Any element of the normalizer $N$ corresponds to an $\F_q$-linear operator
$T$ on $\F_{q^{r+1}}$ such that for some $\beta\in \F_{q^{r+1}}$, $T(\alpha x) = \beta T(x)$
for all $x\in \F_{q^{r+1}}$.  If $Q(x)$ is a polynomial in $\F_q[x]$, $Q(\alpha) = 0$ implies $Q(\beta) = 0$, so $\alpha\mapsto \beta$ defines an element of $\Gal(\F_{q^{r+1}}/\F_q)\cong \Z/(r+1)\Z$.
If the automorphism is trivial, then $T$ corresponds to an element in $Z(A) = A$.
Thus $[N:A] \le r+1$.

On the other hand, any element in $A$ which corresponds to a generator of the field
$\F_{q^{r+1}}$ over $\F_q$ has centralizer $A$, so we may define $F_i$ to be the $i$th
maximal proper subfield of $\F_{q^{r+1}}$ over $\F_q$ and $A_i$ to denote the elements of
$A$ which can be regarded as elements of $F_i$.  The number $n$ of such maximal subfields is the number of distinct prime factors of $r+1$.  Writing
$$r+1 = p_1^{a_1}\cdots p_n^{a_n},$$
with $2\le p_1 < p_2 < \cdots < p_n$, we have
$$(4/\sqrt3)^n \le \frac 2{\sqrt3}  \cdot (r+1).$$
Theorem~\ref{soft} now implies (\ref{Singer}).

If $r\ge 2$ and the characteristic polynomial factors as $P(x) (x-a)$, we can identify
$\F_q^{r+1}$ with $\F_{q^r}\times \F_q$ and the centralizer of $t_2$ with
multiplication operators $(\beta, N(\beta)^{-1})$ on this algebra,
where $\beta\in \F_{q^r}^\times$
and $N$ denotes the norm map to $\F_q$.  The proof now proceeds as before.

\end{proof}

At this point, the reader who is primarily interested in the case $\SL_n(\F_q)$ can proceed directly to \S2.  Here we generalize these examples, replacing $\SL_n$ by any simply connected simple algebraic group $\ul G$.  Our reference for
the theory of such groups is \cite{Car}.

Let $\F_q$ be a finite field and $\ul G$
a simply connected simple algebraic group $\ul G$ over $\F_q$.
Let $G = \ul G(\F_q)$.  It is often useful to think of $\ul G(\F_q)$ as $\ul G(\bar\F_q)^F$,
By the classification of (non-abelian) finite simple groups, with finitely many exceptions,
each such group is an alternating group, a Suzuki or Ree group, or a quotient of
$\ul G(\F_q)$ by its center, for some $\ul G$.
With finitely many exceptions, $G$ determines $\ul G$ and $\F_q$.
Let $\ul B$ be a
Borel subgroup of $\ul G$ and $\ul T_0$ a maximal torus of $\ul B$ defined over $\F_q$.
Let $\ul N$ denote the normalizer of $\ul T_0$ and let $W := \ul N/\ul T_0$ denote
the Weyl group of $\ul G$
with respect to $\ul T_0$.   Thus $F$ acts on $W$ compatibly with its action on the character and cocharacter groups $X$ and $Y$ of $\ul T_0$.
Note that $W^F$ is never trivial, since $\ul B$ defines an $F$-stable longest root $\alpha$,
and the reflection $s_\alpha$ is therefore $F$-invariant.
The group $\ul G$ is determined up to isomorphism by its Dynkin diagram, the prime power $q$,
and the order $m$ of $F$ viewed as an automorphism of the Dynkin diagram of $\ul G$.

Let $W' := W\rtimes \langle \Z/m\Z\rangle$, where $1$ acts on $W$ by $F$.
There is a bijective correspondence
between $G$-conjugacy classes of $\F_q$-rational maximal tori in $\ul G$ and
$F$-conjugacy classes in $W$, i.e., $W$-orbits under the $W$-action $z.w := z^{-1} w F(z)$.
There is also a bijection between $F$-conjugacy classes in $W$ and $W$-orbits in
the coset $W\rtimes 1\subset W'$.
For each $\F_q$-rational torus $\ul T$ of $\ul G$, we write $T = \ul T(\F_q)$ and denote
by $T^{\rcc}$ the set of $G$-conjugacy classes of regular semisimple elements of $T$.

\begin{prop}
\label{rss-bound}
Let $[w']$ denote a $W$-orbit in $W\rtimes 1$ and $\ul T'$ an $\F_q$-rational
maximal torus of $\ul G$
in the class of tori associated with $[w']$.  Let $T' := \ul T'(\F_q)$.  Suppose that $t'\in T'$
is regular semisimple and $\chi$ is an irreducible character of $G$.  Then
$$|\chi(t')| \le (4/\sqrt3)^n |\Stab_W(w')|,$$
where $n$ denotes the number of subgroups of prime order in $\Stab_W(w')$.
\end{prop}

\begin{proof}
As $T'$ contains a regular semisimple element, $\ul T'$ is a non-degenerate torus in the sense
of \cite{Car}.  Let $\ul N'$ denote the normalizer of $\ul T'$.
By \cite[3.6.5]{Car}, there exists an isomorphism $\ul N'/\ul T'\cong W$
mapping $N_G(T')/T'$ to $\Stab_W(w')$.  In particular,
the index of $T'$ in its normalizer equals $|\Stab_W(w')|$.
Let $\langle x_1\rangle,\ldots,\langle x_n\rangle$ denote the elements
of prime order in $N_G(T')/T'$.  Let
$y_i$ denote an element of $N_G(T')$ representing the class $x_i$, and let $A_i\subset T'$
denote the intersection $T'\cap Z_G(y_i)$.
As $\ul G$ is simply connected, Steinberg's theorem \cite[3.5.6]{Car} implies $T' = Z_G(a)$
for every regular semisimple element of $T'$.
We claim that every element $a$ of
$$T' \setminus (A_1\cup \cdots\cup A_n)$$
is regular semisimple and therefore regular.  By \cite[3.5.4]{Car}, if $a$ is not regular, the centralizer $\ul H$ of $a$ in $\ul G$ is a connected reductive group with maximal torus $\ul T'$ such that the Weyl group of $\ul H$ with respect to $\ul T'$ is a non-trivial subgroup of the Weyl group of $\ul G$ with respect to $\ul T'$.  Therefore, there is some non-trivial $F$-stable element of $\ul N'/\ul T'$ which
commutes with $a$.  Some power of this element is of prime order, and therefore some power is of the form $x_i$ for some $i\le n$, and it follows that $a$ belongs to $A_i$.

The proposition now follows immediately from Theorem~\ref{soft}.
\end{proof}

\begin{cor}
\label{nearly-cyclic}
For all $A,B>0$ there exists $C$ such that if $\ul G$ is of rank $r$, $\Stab_W(w')$
is abelian of order less than $Ar^2$ and has a cyclic subgroup of index less than $B$, then
for every regular semisimple $t'\in T'$ and every irreducible character $\chi$,
$$|\chi(t')| \le Cr^3.$$
\end{cor}

\begin{proof}
Let $H := \Stab_W(w')$.
The number of subgroups of $H$ of any given prime order $p$ is bounded
in terms of $B$, and the number of subgroups of order $p>B$ is at most $1$.
On the other hand, the number $n$ of distinct primes $p_i$ dividing $|H|$ is $o(\log |H|)$ since the
sum of the $\log p_i$ is at most $\log |H|$, and if the $p_i$ are arranged in ascending order,
$\log p_i > \log i$.   Thus $(4/\sqrt3)^n = |H|^{o(1)} = r^{o(1)}.$  It follows that for all $\epsilon > 0$,
$$|\chi(t')| = O(r^{2+\epsilon}),$$
which implies the corollary.

\end{proof}

\section{Proof of the Main Theorem}

Let $\delta(G)$ denote the minimum degree of a non-trivial irreducible representation of $G$.
By estimates of Landazuri-Seitz \cite{LaSe}, there exists an absolute constant $c_1>0$ such that
\begin{equation}
\label{ls-bound}
\delta(G) \ge c_1 q^r,
\end{equation}
where $\ul G$ is of rank $r$.

We need one more result before beginning the proof of the main theorem.

\begin{prop}
\label{pairs}
There exist absolute constants $c_2$ and $c_3$ such that
for every sufficiently large group of Lie type $G$ (of rank $r$), there exist maximal tori $\ul T_1$ and $\ul T_2$ such that
for all $g\in G$,
\begin{equation}
\label{independent}
T_1\cap g^{-1} T_2 g = Z(G);
\end{equation}
for every regular $t\in T_1\cup T_2$ and every irreducible character $\chi$ of $G$,
\begin{equation}
\label{char-bound}
|\chi(t)| \le c_2 r^3;
\end{equation}
and
\begin{equation}
\label{many-regular}
|T_i^{\rcc}| \ge \frac{c_3|T_i|}{r^2},
\end{equation}
where $T_i^{\rcc}$ denotes the set of $G$-conjugacy classes of regular semisimple elements of
$T_i$.
\end{prop}

At this point, we prove the proposition only for the case $G = \SL_{r+1}(\F_q)$.
In section 3, we prove it in the remaining cases for which the Frobenius map is standard.
The proposition holds also for the non-standard cases but we omit the proof, since it
is not needed for Theorem~\ref{main}.
Thanks to \cite{FJ} and \cite{GP} and the results of section 4, the conjecture of
Bauer, Catanese, and Grunewald holds for \emph{all} finite simple
groups of Suzuki and Ree type.

\begin{proof}
Let $T_1$ and $T_2$ denote the centralizers of $t_1$ and $t_2$ as defined in
Corollary~\ref{sl-tori}.  We can identify any element $t\in T_1$ with an element $\alpha$ of
$\F_{q^{r+1}}$ of norm $1$ in $\F_q$.  If $\F_q(\alpha) = \F_{q^k}$, then the eigenvalues of $t$ are the conjugates of $\alpha$ over $\F_q$, and each appears with the same multiplicity, $\frac{r+1}k$.
We can identify any element $t\in g^{-1} T_2 g$ with some $\beta\in \F_{q^r}^\times$.
If $\beta$ generates $\F_{q^l}$, the spectrum of $t$ is the multiset union of the
single eigenvalue $N(\beta)^{-1}\in \F_q$ and all $\F_q$-conjugates of $\beta$, each with multiplicity $\frac rl$.  Therefore, every element in $T_1\cap g^{-1} T_2 g$ has a single eigenvalue, which must belong to $\F_q$.  It follows that $T_1\cap g^{-1} T_2 g$ is the group of scalar matrices, i.e., $Z(G)$.

If $r\ge 2$, $t\in T_1\cup T_2$ is regular semisimple,
and $\chi$ is an irreducible character of $G$, then
$$|\chi(t)| \le \frac{2(r+1)^3}{\sqrt 3} \le \frac{16 r^3}{\sqrt 3},$$
by Corollary~\ref{sl-tori}.
For rank $1$, i.e., in the case $\ul G = \SL_2$, we have $|\chi(t)|\le 2$ for all regular elements and all irreducible characters.

The element in $T_1$ corresponding to $\alpha$ is regular if and only if
$\F_q(\alpha) = \F_{q^{r+1}}$.  The set of elements in $\F_{q^{r+1}}$ not generating
the whole field has cardinality less than
$$\sum_{1 < d\mid r+1} |\F_{q^{\frac{r+1}d}}| \le \sum_{i=1}^{(r+1)/2} q^ i \le \frac{q^{(r+1)/2+1}-1}{q-1}
< 2q^{(r+1)/2}.$$
Except in rank $1$, this quantity is less than $q^r/2 < |T_1|/2$ except when
$q^{(r-1)/2} < 4$.  This gives a finite list of exceptions together with the case $r=1$.
For $\SL_2$, a more careful estimate reveals that the number of elements in
$T_1$ which are not regular is at most $2$, so again more than half of the elements in $T_1$
are regular when $|G|$ is sufficiently large.  Now $[N(T_1):T_1]  \le r+1$ implies
$$|T_1^{rcc}| \ge \frac{|T_1|}{2(r+1)} \ge \frac{|T_1|}{4r^2}.$$
Likewise, the element of $T_2$ corresponding to $\beta\in \F_{q^r}$ is regular as long as
$r\ge 2$ and $\F_q(\beta) = \F_{q^r}$.  There are less than
$$\sum_{i=0}^{r/2} q^i < 2q^{r/2}$$
non-regular elements.  For $r=1$, there are at most $2<q^{r/2}$ non-regular elements in $T_2$.
Either way, if $|G|$ is sufficiently large, at least half of the elements of $T_2$ are regular, and
the estimate for $|T_2^{\rcc}|$ follows from $[N(T_2):T_2] \le r$.

\end{proof}

We now prove the main theorem, assuming Proposition~\ref{pairs}, for all finite simple groups of the form $G/Z(G)$, where $G = \ul G(\F_q)$.

\begin{proof}

By \cite{GP}, we may assume that $G$ is not of the form $\SL_2(\F_q)$, $\SL_3(\F_q)$,
or $\SU_3(\F_q)$.

Let $X_i$ denote the union of all conjugacy classes $C_{i,j}$
of regular semisimple elements of $T_i$.
We claim that for $i\in\{1,2\}$, there exist $(x_i,y_i,z_i)\in X_i^3$
such that $x_i y_i = z_i$, and $\langle x_i,y_i\rangle = G$.
Let $\bar x_i,\bar y_i,\bar z_i$ denote the images of
$x_i,y_i,z_i$ respectively in $\bar G := G/Z(G)$.
By (\ref{independent}), we have
$$\Sigma(\bar x_1,\bar y_1)\cap \Sigma(\bar x_2,\bar y_2) = \{e\}.$$

To prove that triples $(x_i,y_i,z_i)$ as above actually exist, we estimate the cardinality
$$N_i := \{(x_i,y_i,z_i)\in X_i^3\mid x_iy_i = z_i\}$$
and compare this to the sum
\begin{equation}
\label{non-maximal}
\sum_{M\in \max G}|\{(x_i,y_i,z_i)\in (X_i\cap M)^3\mid x_iy_i = z_i\}|,
\end{equation}
where $\max G$ denotes the set of maximal proper subgroups of $G$.
This quantity is bounded above by
$$\sum_{M\in \max G} |M|^2 = |G|^2 \sum_{M\in \max G} \frac 1{[G:M]^2}
\le \frac{|G|^2}{m(G)^{1/2}}\sum_{M\in \max G}[G:M]^{-3/2},$$
where $m(G)$ is the minimal index of a proper subgroup of $G$ or, equivalently, the minimal degree of a non-trivial permutation representation of $G$.
By (\ref{ls-bound}), $m(G) \ge c_1 q^r$.
The main theorem of \cite{LMS}
asserts that for all $s>1$,
$$\lim_{|G|\to\infty}\sum_{M\in \max G}[G:M]^{-s}\to 0,$$
where the limit is taken over any sequence of finite simple groups with order tending to $\infty$.
We could equally well take the limit over quasi-simple groups,
since there is an index-preserving bijective correspondence between the
maximal proper subgroups of a quasi-simple group and those of its simple quotient.
This implies that (\ref{non-maximal}) is bounded above by
$q^{-r/2}|G|^2$ for all $|G|$ sufficiently large.

Decomposing $X_i$ into conjugacy classes $C_{i,j}$, indexed by $j\in T^{\rcc}_i$, we can write
$$N_i = \sum_{j,k,l\in T_i^{\rcc}}
|\{(x,y,z)\in C_{i,j}\times C_{i,k}\times C_{i,l}\mid xy = z\}|$$
Using the well-known formula \cite[(7.3)]{Ser} for the summand, this is
\begin{equation}
\label{quad-sum}
\sum_{j,k,l\in T^{\rcc}_i}\frac{|G|^2}{|T_i|^3}\sum_{\chi \in G^*}\frac{\chi(C_{i,j})\chi(C_{i,k})\bar\chi(C_{i,l})}
{\chi(1)}.
\end{equation}
By (\ref{many-regular}), the contribution of the trivial character to (\ref{quad-sum}) is
$$\frac{|G|^2|T^{\rcc}_i|^3}{|T_i|^3}\ge c_4|G|^2 r^{-6},$$
for some constant $c_4$ independent of $G$.  By (\ref{char-bound}), the absolute value of the sum
of the contribution of all non-trivial characters to (\ref{quad-sum}) is bounded above by
\begin{equation}
\label{ineq}
c_5 r^9|G|^2\sum_{\chi\neq 1} \chi(1)^{-1}\le c_5 r^9\delta(G)^{-1/4}|G|^2\sum_{\chi\neq 1} \chi(1)^{-3/4}.\end{equation}

Now, \cite[Corollary 1.3]{LiSh2} asserts that for every $t>1/2$,
$$\lim_{|G|\to\infty} \sum_{\chi}\chi(1)^{-t} = 1,$$
where $G$ ranges over groups of Lie type not of the form $\SL_2(\F_q)$, $\SL_3(\F_q)$, or $\SU_3(\F_q)$.    (Actually, this corollary is stated in a form where $G$ is assumed to be simple, but the proof, an immediate application of Theorems 1.1 and 1.2 of the same paper, holds equally when $G$ ranges over quasi-simple groups, as in this case.)
Having already excluded $\SL_2(\F_q)$, $\SL_3(\F_q)$,and $\SU_3(\F_q)$, we may set
$t=3/4$, and conclude that
$$\lim_{|G|\to\infty} \sum_{\chi\neq 1}\chi(1)^{-3/4} = 0.$$
From this and (\ref{ls-bound}), we deduce that for $G$ sufficiently large, the right hand side of
(\ref{ineq}) is bounded above by
$$r^9q^{-r/4} |G|^2,$$
and it follows that (\ref{quad-sum}) is bounded below by
$$c_6 r^{-6} |G|^2$$
for $G$ sufficiently large.  Comparing this to our upper bound for (\ref{non-maximal}), we conclude that for all $G$ sufficiently large, there exist triples $(x_i,y_i,z_i)$ with the properties claimed.

\end{proof}

\section{Pairs of Maximal Tori}

In this section, we prove Proposition~\ref{pairs} in the case that $F$ is a standard Frobenius map.
\begin{proof}

For each choice of Dynkin diagram $\Delta$ and each positive integer $m$ dividing the order of
$\Aut(\Delta)$, we specify two elements $w_1,w_2\in W\rtimes 1$.
We then prove that if $\ul T_1$ and $\ul T_2$ are $\F_q$-rational maximal tori of $\ul G$
corresponding to $w_1$ and $w_2$ respectively,
then $T_i := \ul T_i(\F_q)$ satisfy (\ref{independent}), (\ref{char-bound}), and (\ref{many-regular}) .
Instead of proving (\ref{char-bound}) directly, we deduce it from Proposition~\ref{rss-bound} and the inequalities
\begin{equation}
\label{small-normalizer}
[N_i:T_i] \le  c_7 r^2,
\end{equation}
and
\begin{equation}
\label{few-max}
n_i \le \log_4 r + c_8,
\end{equation}
where $N_i$ is the group of $\F_q$-points of the normalizer $\ul N_i$ of $\ul T_i$,
and $n_i$ denotes the number of subgroups of prime order in $\Stab_W(w_i)$.

We note that
$T_i$ will always contain at least one regular semisimple element and will therefore
be non-degenerate, so if $c_7$ and $c_8$ exceed $|W|$,
(\ref{small-normalizer}) and (\ref{few-max}) hold trivially.
We will only check these conditions, therefore, in the classical cases,
i.e., $A_r$, ${}^2A_r$, $B_r$, $C_r$, $D_r$, and ${}^2 D_r$.

By \cite[3.2.2]{Car}, $|T| = |Y/(F-1)Y|$, where $Y$ denote the cocharacter group of $\ul T$.
If $\ul T$ is associated to the orbit of $(w,1)$
in $W\rtimes 1$, this equals the determinant of $q w$ acting
on the coweight space (or equivalently on the weight space).  The eigenvalues of $(w,1)$ are
roots of unity, and decomposing the spectrum into Galois orbits, we obtain a formula for
$|T|$ of the form $\prod_{j=1}^s \Phi_{k_j}(q)$, where $\Phi_k$ denotes the $k$th cyclotomic polynomial, and the total degree of the product is $r$.
The subset of elements of $T$ fixed by any particular non-trivial
element of $W$ is a polynomial in $q$ of lower degree.  As every semisimple element in $T$ which is not regular is fixed by some non-trivial element of $W$,
it follows that in fixed rank,
(\ref{many-regular}) holds for all sufficiently large $G$.

To check (\ref{independent}) it is useful to note that for monic polynomials $R(x)$ and $S(x)$,
the greatest common denominator of $R(q)$ and $S(q)$ divides the resultant of $R$ and $S$.
Also, for cyclotomic polynomials
$$\res(\Phi_a(x),\Phi_b(x)) =
\begin{cases}
\ell&\text{if $b/a = \ell^k$, $\ell$ prime, $k\in\Z\setminus\{0\}$.}\\
1&\text{otherwise.}
\end{cases}$$
To compute the resultant of two products of cyclotomic polynomials, we use
bimultiplicativity.  In general, the greatest common divisor
$(R(q),S(q))$ divides $(\res(R(x),S(x)),R(q))$.
Since $T_1\cap g^{-1} T_2 g$ contains $Z(G)$ for all $g$, it suffices to prove that
$(R(q),S(q))$ divides $|Z(G)|$ to prove that $T_1\cap g^{-1} T_2 g = Z(G)$ for all $g\in G$.

We now specify choices for $w_i$ in each case, except split $A_r$, which has already been treated, checking (\ref{many-regular}), (\ref{small-normalizer}), and (\ref{few-max})
in the classical cases and (\ref{independent}) in all cases.
Note that (\ref{many-regular}) follows from (\ref{small-normalizer}) as long as at least
one quarter of the elements of $T_i$ are regular semisimple, which we will see is in fact the case whenever $G$ is sufficiently large.
By Corollary~\ref{nearly-cyclic}, (\ref{few-max}) follows from (\ref{small-normalizer}) as long
as $\Stab_W(w_i)$ is abelian with a cyclic subgroup of bounded order.

\case{$G = {}^2A^{\sc}_r(q) = \SU_{r+1}(\F_q)$.}
In this case, $W = S_{r+1}$, and $m=2$.  The action of $F$ is inner
(it is in fact conjugation by a product of $\lfloor (r+1)/2\rfloor$ disjoint $2$-cycles), so
$W' = S_{r+1}\times \Z/2\Z$, and $S_{r+1}$-orbits of elements of $S_{r+1}\times 1$
can be identified with $S_{r+1}$-conjugacy classes. We set
$$w_1 = ((1\,2\,3\,\ldots\,r+1),1),\ w_2 = ((1\,2\,3\,\ldots\, r),1).$$
The stabilizers of $w_1$ and $w_2$ are cyclic groups of order $r+1$ and $r$ respectively, just as in the split case.  This gives (\ref{small-normalizer}) and (\ref{few-max}).

We have
\begin{equation*}
\begin{split}
(|T&_1|,|T_2|) = (R(q),S(q)(q+1)) \\
& = (R(q),S(q))(R(q),q+1)
 | \res(R(x),S(x))(\res(R(x),x+1),q+1) \\
& = (r+1,q+1),
\end{split}
\end{equation*}
where
$$R(x) = \frac{x^{r+1}-1}{x-1},\ S(x) = \frac{x^r-1}{x-1}.$$
Thus, $(|T_1|,|T_2|)$ divides the order of $Z(G)$, so $T_1\cap g^{-1} T_2 g = Z(G)$.
This gives (\ref{independent}).

The analysis of non-regular elements of $T_1$ or $T_2$ works as in the split case, giving an $O(q^{(r+1)/2})$ upper bound,
so again more than half of the elements are regular when $G$ is large.  This implies (\ref{many-regular})

\case{$G = B_r^{\sc}(q) = \Spin_{2r+1}(\F_q)$.}
Here the Weyl group is $(\Z/2\Z)^r\rtimes S_r$, i.e., the group of permutations
of the set $\{1,2,\ldots,r,1',2,'\ldots,r'\}$ which respects the partition
$\{\{1,1'\},\ldots,\{r,r'\}\}$.  We set
$$w_1 = (1\,2\,\ldots\,r\,1'\,2'\,\ldots\,r'),\ w_2 = (1\,2\,\ldots\,r)(1'\,2'\,\ldots\,r').$$
The centralizers of $w_1$ and $w_2$ are $\Z/2r\Z$ and $\Z/2\Z\times \Z/r\Z$ respectively, so
again we have (\ref{small-normalizer}) and (\ref{few-max}).  The torus orders are
$$|T_1| = q^r+1,\ |T_2| = q^r-1,$$
so $(|T_1|,|T_2|)$ divides $(2,q^r-1) = (2,q-1) = |Z(G)|$.

We can identify $T_i$ with $\{\lambda\in \bar\F_q\mid \lambda^{|T_i|}=1\}$,
and if $t\in T_i$ fails to be regular, the corresponding element $\lambda$ satisfies
$\lambda^{q^k\pm 1} = 1$ where $p^k\pm 1\mid |T_i|$.  Thus, we may assume $k\le r/2$,
and the number of non-regular elements of $T_1$ or $T_2$ is $O(q^{r/2})$.

\case{$G = C_r^{\sc}(q) = \Sp_{2r}(\F_q)$.}
This case is in all respects parallel to that of $B_r^{\sc}(q)$.

\case{$G = D_r^{\sc}(q) = \Spin^+_{2r}(\F_q)$.}
We regard $W$ as an index $2$ subgroup of $(\Z/2\Z)^r\rtimes S_r$.
If $r$ is divisible by $4$, we set $s=r/2$ and let
\begin{align*}
w_1 &= (1\,2\,\ldots\,s-1)(s\,\ldots\,r)(1'\,2'\,\ldots\,(s-1)')(s'\,\ldots\,r') \\
w_2 &= (1\,2\,\ldots\,(s-1)\,1'\,2'\,\ldots\,(s-1)')(s\,\ldots\,r\,s'\,\ldots\,r').
\end{align*}
The centralizers of $w_1$ and $w_2$ are both contained in
$$\Z/(2s-2)\Z\times \Z/(2s+2)\Z\cong \Z/(s^2-1)\Z\times (\Z/2\Z)^2,$$
which has order $O(r^2)$ and
a cyclic subgroup of index $4$.
We have
$$|T_1| = (q^{s-1}-1)(q^{s+1}-1),\ |T_2| = (q^{s-1}+1)(q^{s+1}+1),$$
so $(|T_1|,|T_2|)$ divides
\begin{align*}\res(x^{s-1}&-1,x^{s-1}+1)\res(x^{s-1}-1,x^{s+1}+1) \\
&\cdot \res(x^{s+1}-1,x^{s-1}+1)\res(x^{s+1}-1,x^{s+1}+1) =16,
\end{align*}
but also the highest power of $2$ dividing $(|T_1|,|T_2|)$ is $4$ if $q$ is odd and $1$
if $q$ is even.  Thus, $(|T_1|,|T_2|) = |Z(G)|$.

If $r\equiv2$ (mod $4$), we set
\begin{align*}
w_1 &= (1\,2\,\ldots\,r)(1'\,2'\,\ldots\,r'),\\
w_2 &= (1\,2\,1'\,2')(3\,4\,\ldots\,r\,3'\,4'\,\ldots\,r').
\end{align*}
The centralizers of $w_1$ and $w_2$ are $\Z/r\Z\times \Z/2\Z$
and $\Z/(2r-2)\Z\times \Z/4\Z$ respectively.  Each has order $O(r)$ and a cyclic subgroup of index $\le 4$.  We have
$$|T_1| = q^r-1,\ |T_2| = (q^{r-2}+1)(q^2+1),$$
so $(|T_1|,|T_2|)$ divides
$$(q^r-1,q^{r-2}+1)(q^r-1,q^2+1) = |Z(G)|.$$

If $r$ is odd, we set
\begin{align*}
w_1 &= (1\,2\,\ldots\,r)(1'\,2'\,\ldots\,r'),\\
w_2 &= (1\,1')(2\,3\,\ldots\,r\,2'\,3'\ldots\,r').
\end{align*}
The centralizer of $w_1$ is $\Z/r\Z$,
while the centralizer of $w_2$ is $\Z/(2r-2)\Z$.   We have
$$|T_1| = q^r-1,\ |T_2| = (q^{r-1}+1)(q+1),$$
so $(|T_1|,|T_2|)$ divides
$$(q+1,q^r-1)(q^{r-1}+1,q^r-1) =  (2,q+1)^2 = |Z(G)|.$$

With one exception, we can bound the number of non-regular elements of $T_i$
by $O(q^{(s-1)/2})$.  This exception is $T_2$ in the case that $r$ is $2$ (mod $4$).
Here elements of $T_2$ can be regarded as pairs $(\alpha,\beta)\in \bar\F_q^2$
such that $\alpha^{q^2+1} = \beta^{q^{r-2}+1} = 1$.
A necessary condition that such a pair correspond to a non-regular element is that
$\alpha^{q^a\pm 1} = 1$ for some positive $q^a\pm 1 < q^2+1$ or
$\beta^{q^b\pm 1} = 1$ for some positive $q^b\pm 1 < q^{r-2}+1$.  The former condition can
hold for as many as $\frac 2{q^2+1} |T_2| \le (2/5)|T_2|$ pairs, while the latter condition
holds for $O(q^{(r+2)/2})$ pairs.  We conclude that if $G$ is sufficiently large, it is still true that
more than half of the elements of $T_2$ are regular.

\case{$G = {}^2D_r^{\sc}(q) = \Spin^-_{2r}(\F_q)$.}
We identify $W\rtimes \Z/2\Z$ with $(\Z/2\Z)^r\rtimes S_r$.  If $r$ is even, we set
\begin{align*}
w_1 &= (1\,2\,\ldots\,r\,1'\,2'\,\ldots\,r'),\\
w_2 &= (1\,1')(2\,3\,\ldots\,r)(2'\,3'\ldots\,r').
\end{align*}
The $W$-stabilizers of $w_1$ and $w_2$  are contained in their centralizers in
$(\Z/2\Z)^r\rtimes S_r$, which are respectively
$\Z/2r\Z$ and $\Z/(r-1)\Z\times \Z/2\Z$.
Each has order $O(r)$ and a cyclic subgroup of index $\le 2$.
We have
$$|T_1| = q^r+1,\ |T_2| = (q^{r-1}-1)(q+1),$$
so $(|T_1|,|T_2|)$ divides
$$(q+1,q^r+1)(q^{r-1}-1,q^r+1) = (2,q+1)^2 = |Z(G)|.$$
There are $O(q^{r/2})$ non-regular elements in $T_1$ and at most
$\frac {2}{q+1}|T_2| +O(q^{(r+1)/2})$ non-regular elements in $T_2$.
IF $G$ is sufficiently large, more than one quarter of the elements of $T_i$ are regular.

If $r$ is odd, we set
\begin{align*}
w_1 &= (1\,2\,\ldots\,r\,1'\,2'\,\ldots\,r'),\\
w_2 &= (1\,2\,1'\,2')(3\,\ldots\,r)(3'\ldots\,r').
\end{align*}
The centralizers of $w_1$ and $w_2$ in $(\Z/2\Z)^r\rtimes S_r$ are respectively
$\Z/2r\Z$ and $\Z/(r-2)\Z\times \Z/4\Z$, each of order $O(r)$ with a cyclic subgroup of
index $\le 4$.  We have
$$|T_1| = q^r+1,\ |T_2| = (q^{r-2}-1)(q^2+1),$$
so $(|T_1|,|T_2|)$ divides
$$(q^2+1,q^r+1)(q^{r-1}-1,q^r+1) = (2,q^2+1)^2 = (2,q+1)^2 = |Z(G)|.$$
There are $O(q^{r/2})$ non-regular elements in $T_1$ and at most
$\frac2{q^2+1}|T_2|+O(q^{(r+2)/2})$ non-regular elements in $T_2$.
If $G$ is sufficiently large, more than half of the elements of $T_i$ are regular.

\case{$G = {}^3D^{\sc}_4(q)$.}
By \cite[Table 9]{Spr}, there exists elements $w_1$ and $w_2$ such that
$$|T_1| = \Phi_1(q)^2\Phi_2(q)^2,\ |T_2| = \Phi_{12}(q).$$
so $(|T_1|,|T_2|)$ divides
$$\res(\Phi_1(x)^2\Phi_2(x)^2,\Phi_{12}(x))=1.$$

\case{$G = E_6^{\sc}(q)$.}
By \cite[Table 1]{Spr}, there exists elements $w_1$ and $w_2$ such that

$$|T_1| = \Phi_1(q)\Phi_2(q)\Phi_8(q),\ |T_2| = \Phi_9(q),$$
so $(|T_1|,|T_2|)$ divides
\begin{align*}
(\res(\Phi_1(x)\Phi_2(x)\Phi_8(x),\Phi_9(x))&,\Phi_9(q))=(3,\Phi_9(q)) \\
&= (3,q-1) = |Z(G)|.
\end{align*}
\case{$G = {}^2E_6^{\sc}(q)$.}
By \cite[Table 1]{Spr}, there exists elements $w_1$ and $w_2$ such that
$$|T_1| = \Phi_1(q)\Phi_2(q)\Phi_8(q),\ |T_2| = \Phi_{18}(q),$$
so $(|T_1|,|T_2|)$ divides
\begin{align*}
(\res(\Phi_1(x)\Phi_2(x)\Phi_8(x),\Phi_{18}(x)&,\Phi_{18}(q))=(3,\Phi_{18}(q)) \\
&= (3,q+1) = |Z(G)|.
\end{align*}

\case{$G = E_7^{\sc}(q)$.}
By \cite[Table 2]{Spr}, there exists elements $w_1$ and $w_2$ such that
$$|T_1| = \Phi_1(q)\Phi_9(q),\ |T_2| = \Phi_2(q)\Phi_{14}(q),$$
so $(|T_1|,|T_2|)$ divides
\begin{align*}
(\res(\Phi_1(x)\Phi_9(x)&,\Phi_2(x)\Phi_{14}(x))=(2,\Phi_1(q)\Phi_9(q)) \\
&= (2,q-1) = |Z(G)|.
\end{align*}

\case{$G = E_8(q)$.}
By \cite[Table 3]{Spr}, there exists elements $w_1$ and $w_2$ such that
$$|T_1| = \Phi_{24}(q),\ |T_2| = \Phi_{30}(q),$$
so $(|T_1|,|T_2|)$ divides
$$\res(\Phi_{24}(x),\Phi_{30}(x)) = 1.$$

\case{$G = F_4(q)$.}
By \cite[Table 4]{Spr}, there exists elements $w_1$ and $w_2$ such that
$$|T_1| = \Phi_8(q),\ |T_2| = \Phi_{12}(q),$$
so $(|T_1|,|T_2|)$ divides
$$\res(\Phi_8(x),\Phi_{12}(x)) = 1.$$

\case{$G = G_2(q)$.}
Let $w_1$ and $w_2$ denote rotations of $\pi$ and $2\pi/3$ respectively.  Then
$$|T_1| = \Phi_2(q)^2,\ |T_2| = \Phi_3(q),$$
so $(|T_1|,|T_2|)$ divides
$$\res(\Phi_2^2(x),\Phi_3(x)) = 1.$$

\end{proof}

\section{Ree groups}

In this section we finish the proof of Theorem~\ref{main} by
treating the family of Ree groups ${}^2F_4(2^{2f+1})$.
The other family of Ree groups ${}^2G_2(3^{2f+1})$ and the Suzuki
groups ${}^2B_2(2^{2f+1})$ can be treated similarly. However,
alternative proofs for these groups have already been provided in
\cite{FJ} and \cite{GP}.

The proof for the Ree groups $G={}^2F_4(2^{2f+1})$ relies on the
following two theorems.

\begin{thm}\cite{Mal}. \label{Mal}
The group $G={}^2F_4(q)$, where $q=2^{2f+1}$ ($f \geq 1$), contains
two (cyclic) maximal tori $T_1$ and $T_2$ whose orders $\tau_1$ and
$\tau_2$ are given by
\begin{align*}
   \tau_1&= q^2+\sqrt{2q^3}+q+\sqrt{2q}+1
   = 2^{4f+2} + 2^{3f+2} + 2^{2f+1} + 2^{f+1} +1 ,\\
   \tau_2&= q^2-\sqrt{2q^3}+q-\sqrt{2q}+1
   = 2^{4f+2} - 2^{3f+2} + 2^{2f+1} - 2^{f+1} + 1.
\end{align*}
Moreover,
\[
    N_G(T_1) \cong Z_{\tau_1} \rtimes Z_{12} \text{ and }
    N_G(T_2) \cong Z_{\tau_2} \rtimes Z_{12}
\]
are maximal subgroups in $G$, and they are the only maximal
subgroups containing $T_1$ and $T_2$ respectively.
\end{thm}

Recall that $Z_n$ denotes a cyclic group of order $n$.

\begin{thm}\cite{Gow}. \label{Gow}
Let $G$ be a finite simple group of Lie type, and let $z \neq 1$ be
a semisimple element in $G$. Let $C_1$ and $C_2$ be any conjugacy
classes of $G$ consisting of regular semisimple elements. Then $z$
is expressible as $z=xy$, where $x \in C_1$ and $y \in C_2$.
\end{thm}

We recollect a few observations on the orders $\tau_1$ and $\tau_2$.

\begin{lem}\label{lem.t1.t2}
The orders $\tau_1$ and $\tau_2$ appearing in Theorem~\ref{Mal}
satisfy:
\begin{enumerate}\renewcommand{\theenumi}{\it \roman{enumi}}
\item $\tau_1 \equiv 1 \pmod {12}$ and $\tau_2 \equiv 1 \pmod {12}$.
\item If $f \geq 2$, then $\phi(\tau_1), \phi(\tau_2) \geq 12 \cdot 13 = 156$, where $\phi$ denotes
the Euler function.
\item $\tau_1$ and $\tau_2$ are relatively prime.
\end{enumerate}
\end{lem}
\begin{proof}
\begin{enumerate}\renewcommand{\theenumi}{\it \roman{enumi}}
\item Observe that when $f \geq 1$,
\begin{align*}
   \tau_1-1 &= 2^{4f+2} + 2^{3f+2} + 2^{2f+1} + 2^{f+1}, \text{ and}\\
   \tau_2-1 &= 2^{4f+2} - 2^{3f+2} + 2^{2f+1} - 2^{f+1},
\end{align*}
are both divisible by $3$ and by $4$, hence $\tau_1-1 \equiv 0
\pmod{12}$ and $\tau_2-1 \equiv 0 \pmod{12}$.

\item It is well-known that $\phi(n) \geq \sqrt{n}$ for any $n>6$ (see~\cite{SMC}, pp. 9).
If $f \geq 4$ then $\tau_1 \geq \tau_2 \geq 246241 > 400^2$ and the
result immediately follows. For $f=2$, $\tau_1 = 1321$ and $\tau_2 =
793$, and so $\phi(\tau_1) = 1320$ and $\phi(\tau_2) = 720$. For
$f=3$, $\tau_1=18577$ and $\tau_2=14449$, and so $\phi(\tau_1) =
17136$ and $\phi(\tau_2) = 14448$.

\item Both $\tau_1$ and $\tau_2$ are odd, and if some odd prime divides both of them,
then it also divides $\tau_1-\tau_2=2\sqrt{2q}(q+1)=2^{f+2}(q+1)$,
and so it must divide $(q+1)$. Now, since
\[
   \tau_1 = q^2+(q+1)(1+\sqrt{2q}),
\]
such a prime also divides $q^2=2^{4f+2}$, yielding a contradiction.
\end{enumerate}
\end{proof}

Let $T$ be one of $T_1,T_2$ and set $\tau=|T|$.

Since $T$ is cyclic, there exists some $z \in T$ of exact order
$\tau$. Note that $z$ is regular and semisimple, and denote by $C_z$
the conjugacy class of $z$ in $G$. By Theorem~\ref{Mal},
\begin{align*}
|C_z \cap T| &= |\{t \in T: t=gzg^{-1} \text{ for some } g \in G\}|
\\ &= |\{gzg^{-1}: g\in N_G(T)\}| \leq [N_G(T):T] =12.
\end{align*}
Moreover, $|z|=\tau$, $|N_G(T)|=12\cdot \tau$ and $\gcd(\tau,12)=1$,
therefore $C_z \cap N_G(T) \subseteq (C_z \cap T)$ and so $|C_z \cap
N_G(T)| \leq 12$.

By Lemma~\ref{lem.t1.t2}{\it(ii)}, if $f \geq 2$ then $T$ contains
at least $156$ elements of exact order $\tau$, and each of them has
at most $12$ $G$-conjugates in $T$, thus one can find $13$ elements
$w_1,\dots,w_{13} \in T$ of exact order $\tau$, each of them
belonging to a different conjugacy class in $G$. These elements are
regular and semisimple, and we denote their corresponding conjugacy
classes by $C_{w_1},\dots,C_{w_{13}}$.

By Theorem~\ref{Gow}, one can write $z=x_jy_j$,where $x_j \in C_z$
and $y_j \in C_{w_j}$, for $1\leq j \leq 13$. Observe that all the
elements $y_1,\dots,y_{13}$ are necessarily distinct, as they belong
to distinct conjugacy classes. Therefore, all the elements
$x_1,\dots,x_{13}$ are also distinct.

Now, since $|C_z \cap N_G(T)| \leq 12$, there exists some $1\leq
j\leq 13$ such that $x_j \notin N_G(T)$. Since $T=\langle z \rangle$
and $N_G(T)$ is the only maximal subgroup of $G$ containing $T$ (by
Theorem~\ref{Mal}) it follows that $\langle x_j,z \rangle = \langle
x_j,y_j \rangle = G$.

\medskip

Applying the above argument for $T=T_1$ and $T=T_2$, we deduce that
whenever $f\geq 2$, there exist $x_1,y_1,z_1$ of exact order
$\tau_1$ and $x_2,y_2,z_2$ of exact order $\tau_2$ such that
\begin{align*}
& x_1y_1z_1 = 1 = x_2y_2z_2, \text{ and } \\
& \langle x_1, y_1 \rangle = G = \langle x_2, y_2 \rangle.
\end{align*}
One can verify that this statement also holds for $f=1$ using
the \textsc{Magma} computer program \cite{Magma}.

Moreover, by Lemma~\ref{lem.t1.t2}{\it(iii)}, $\tau_1$ and $\tau_2$
are relatively prime, therefore
\[
    \Sigma(x_1,y_1)\cap \Sigma(x_2,y_2) = \{e\},
\]
implying that $(x_1,y_1, x_2,y_2)$ is an unmixed Beauville structure
for the group $G={}^2F_4(2^{2f+1})$.


\end{document}